\documentclass[12pt,twoside]{article}

\usepackage{setspace}
\usepackage{amssymb}
\usepackage{amsthm}
\usepackage{graphicx,color}
\usepackage{pstricks}
\usepackage{pst-text}
\usepackage{pst-tree}
\usepackage{multido}
\def\R{{\mathbb R}}
\def\N{{\mathbb N}}
\def\C{{\mathbb C}}

\newtheorem{thm}{Theora}[section]
\newtheorem{theo}[thm]{Theorem}
\newtheorem{hypothesis}[thm]{Hypothesis}

\newtheorem{rem}[thm]{Remark}

\newcommand\dint{\displaystyle\int}
\newcommand\ds{\displaystyle\sum}

\newcommand\dlim{\displaystyle\lim}

\psset{xunit=1cm,yunit=1cm,runit=1cm,shortput=tab}
\begin{document}
\pagestyle{myheadings} \markboth{ Albeverio.S and Smii.B }{Borel summation of the heat kernel} \thispagestyle{empty}

\title{Borel summation of the small time expansion of some SDE's driven by Gaussian
white noise}
\author{Sergio Albeverio\footnote{Dept. Appl. Mathematics, University of Bonn, HCM, BiBoS, IZKS; KFUPM Dhahran;  CERFIM, Locarno. {\tt albeverio@uni-bonn.de}}\,\,\, %
and\,\,
  Boubaker Smii\footnote{ King Fahd University of
Petroleum and Minerals, Dept. Math. and Stat., Dhahran 31261, Saudi
Arabia. \hspace*{5mm}{\tt boubaker@kfupm.edu.sa}}}

\maketitle
 \begin{abstract} 
We consider stochastic differential equations with a drift term of gradient type and driven
by  Gaussian white noise on $\R^d$. 
Particular attention is given to the kernel  $p_t,\,t\geq 0$ of the transition semigroup associated with the solution process.\\
 Under some rather strong regularity and growth assumptions on the coefficients, we adapt previous work by Thierry Harg\'e on Schr\"{o}dinger operators and prove that the small time asymptotic expansion of $p_t,\,t> 0$ is Borel summable.\\
 We also briefly indicate some extensions and applications.

\end{abstract}
\vskip0.2cm

\section{Introduction and motivation}
In the description of the dynamics of many phenomena studied in natural sciences, engineerings as well as in economics and social sciences the use of stochastic differential equations has been very extensive, in order to cope both with deterministic and random effects. Very often solutions of these equations are obtained by complicated constructions and it is difficult to extract more concrete information about related quantities of interest. Then one has recourse to numerical or computational methods, that involve suitable approximations, or,often alternatively, one develops asymptotic methods, like expansions with respect to parameters or variables appearing in the equations.\\
In the present paper we shall look at one particular type of expansion, namely small time expansions for stochastic differential equations for diffusion process, with Brownian motion as driving process.\\
The usefulness of such expansions has been pointed out very early in many contexts, including differential geometry, analysis, quantum mechanics and statistical mechanics, see, e.g. \cite{ACDP, AlDiPMa, ALDIBOUB, ALDIBOUB1, AGoWu, AHk, ALEBOU, ALRS, Az, Ba, BE, BBF, CL, Co1, Co2, GHLOW, GULI, HkPS,  JT, Kan, Ki, Kol,  Kuso, Lu, Mo, TA, Ta, Tub, UMN, UM} and references therein.\\ 
Applications to various areas including physics, neurobiology, see, e.g. \cite{ABrMa, AlGYo, AKKR, GST, GS, GOS, MM, Ta} and mathematical finance, see,e.g. \cite{GULI} have been also provided. For the latter area let us mention particularly that besides well known models discussed in various textbooks, also models with  inclusion of interacting assets with different types of noise have also been studied in \cite{ALiMa, ALBO, ASte, Lu}.\\
Advanced statistical methods have also been applied to finance in different directions, from econometrics to methods inspired by statistical mechanics, see, e.g., \cite{Ma}, \cite{GaSc}.\\
From a mathematical point of view we shall
develop here first results concerning the study of important properties of
solutions for a class of  SDEs driven by an additive Gaussian white
noise. At the end of the paper we mention a couple of applications.\\
Here we shall concentrate on the class of SDEs of the form:
\begin{equation}\label{e12}
\left\{ \begin{array}{ll} dX_{t}=\beta (X_{t})\,dt+\sigma (X_{t})\,
d\,W_{t}&,\,\, t\geq 0\\ X_0= x^0,\,\,\,x^0\,\in \, \R^d,
\end{array}
\right.
\end{equation}
where $X_t=(X_t)_{t\geq 0}$ is a stochastic process taking values in
a finite dimensional space $\R^d$,  while $\beta$ is a vector field from\,$\R^d$ to $\R^d$ and $\sigma$ is a $d\,\times\,d$-matrix with real coefficients depending on the space variable in $\R^d$, $W_t$ is a standard Brownian motion (Wiener process in $\R^d$).\\
Under suitable regularity and growth assumptions on $\beta$ and $\sigma$, existence and
uniqueness of solutions are known, see, e.g. \cite{GIS, OKS}. \\
Moreover, under additional regularity assumptions the transition semigroup of the solution process $X_t, t\geq 0,$ has a density (with respect to Lebesgue measure on $\R^d),\, p_t(x,y)\,,t\geq 0, x,y\,\in\,\R^d.$ For short we recall  that the transition kernel $p_t$ satisfies then the following Kolmogorov forward equation (Fokker-Planck equation)(the space derivatives being understood with respect to $x$).
\begin{equation}\label{e13}
\left\{ \begin{array}{ll} {d \over \partial t}p_{t}(x,y)=\,\Delta[a(x) \,p_t(x,y)]-\nabla[\beta(x)\,p_t(x,y)]&,\,\, t > 0\\ \dlim_{x \longrightarrow 0}p_0(x,y)\,dx=\delta_0(x-y),
\end{array}
\right.
\end{equation}
 with $a:={1 \over 2}\,\sigma\,\sigma^t,\,\delta_0$ being Dirac measure at the origin, see, eg., \cite{GIS}, \cite{OKS}.\\
 In this paper we concentrate on the case where $\sigma$ is proportional to the unit $d\times d$-matrix\,${\bf 1}$. Only in the last section we consider the case with general $\sigma$ but $d=1$, reducing it by a suitable time change transformation to the case $\sigma=1$. \\
 We discuss the asymptotic expansion of $p_t$ for $t\downarrow\,0$ and its Borel summation.
We  rely essentially on a setting and results contained in important work done by Thierry Harg\'e \cite{TH, TH1, TH2, TH3, TH4} for the study of the Schr\"{o}dinger equation for analytic potentials and its corresponding diffusion equation. For this we need to assume throughout our paper that $\beta$ is a gradient field.\\
Let us now describe the single sections of this paper:\\
In the section 2 we will introduce the basic technical assumptions, provide definitions of Borel summability and present the important results obtained by T. Harg\'e.\\
Section 3 will be reserved to the study of Borel summation of the small time asymptotic expansion of the above transition kernel $p_t$.\\
Section 4 is devoted to some comments on extension to the case $\sigma\neq 1 $, at least for $d=1$, and on some applications.\\
\section{Assumptions, definitions and the results by T. Harg\'e}
In this section we will outline the assumptions needed for adopting to our setting results obtained by T. Harg\'{e}, see. \cite{TH, TH1, TH2, TH3, TH4} for the time dependent Schr\"{o}dinger equation.
\begin{hypothesis}\label{hp:A+F}
Let us consider "potential functions" $V$ from $\R^d$ into $\R$ of the form $V(x):=-\dint e^{i\,x\,\xi}\, \mu_V(d\xi),\,x\,\in\,\R^d$, for some complex Borel measure $\mu_V$ on $\R^d$ such that $\mu_V(-A)= \mu_V(A)$ for any Borel subset $A$ of $\R^d$ (this implies indeed that $V$ is real valued: this property will be indicated by saying that $\mu_V$ is symmetric).\\
(Here and in the whole paper $\R^d$ is Euclidean $d-$dimensional space for any natural number $d$.)  
        \begin{enumerate}
        \item[]
             \item $V$ is assumed moreover to be such that $-\Delta +V \geq E_0$ for some constant $E_0\geq 0$ and there exists a strictly positive $C^1(\R^d)$ function $ \,\phi\,\in\, L^2(\R):=L^2(\R, dx),$ satisfying $(-\Delta +V)\phi=E_0\,\phi$ for some $E_0\geqslant 0.$ We can always assume that $E_0=0,$ otherwise we would change $V$ in $\tilde{V}=V-E_0$ (with corresponding $\mu_{\tilde{V}}(dx)=(\mu_V-E_0\,\delta_0)(dx)$, with $\delta_0(dx)$ the Dirac's measure at zero). In this case, i.e., when  $E_0=0,\, -\Delta +V$ is  non negative on $C^2_0(\R^d)$ in $L^2(\R^{d})$. \\
           
         \item Let $\beta:\R^d\,\rightarrow\,\R^d$ be  such that $\beta(x):=2\,\phi^{-1}(x)\,\nabla\,\phi,$  then the unitary map $U_{\phi}$ from $L^2(\R^{d}, \phi^2\,dx)$ to $ L^2(\R^d) $ given by multiplication by $1 \over {\phi}$ transforms the operator $H=-\Delta-\beta.\nabla$ acting (on the dense domain) $\phi^{-1} C^2_0(\R^d) $ in $L^2(\R^{d}, \phi^2\,dx) $ into the operator  $-\Delta\,+V$ acting (on the dense domain) $ C^2_0(\R^d) $ in $\,L^2(\R^{d},\,dx):=L^2(\R^{d})$, in the sense that
                
         \begin{equation}
         H\,\phi^{-1}\,f=\,\phi^{-1}(-\Delta+V)\,f\,\,\,\,\,\,\forall\,f\,\in\,C^2_0(\R^d).
         \end{equation}
   This transformation is called ground state transformation, see, e.g., \cite{AHKS, Al, AlMa, RS}      
      \item Under some regularity assumptions on $V,$ and consequently on  $\phi,\,$ $-\Delta\,+V\,$ is essentially self adjoint and positive on $C_0^2(\R^d)$ in $L^2(\R^{d})$, and correspondingly for $H$ on $ \phi^{-1} C^2_0(\R^d)$ in $L^2(\R^{d}, \phi^2\,dx)$. Hence their closures, called again $H$ resp. $-\Delta+V $, are generators of strongly continuous contraction semigroups $e^{-tH}$ resp. $e^{-t(-\Delta+V)}$ in $L^2(\R^{d}, \phi^2\,dx)$ resp. $L^2(\R^{d})$  and 
          \begin{equation}
                   e^{-tH}\,\phi^{-1}\,f=\,\phi^{-1}(e^{-t(-\Delta+V)})f\,\,\,\,\,\,\forall\,f\,\in\,C^2_0(\R^d), t \geq 0.
          \end{equation}
  \end{enumerate}           
 \end{hypothesis}
  The relation between $V$ and $\phi$ can be expressed through $\beta$ by the formula
  \begin{equation}
  V={1 \over 2}\,({\mid\beta\mid^2 \over 2}+\,div \beta)\,\,.
  \end{equation}
 In the following concepts like asymptotics expansions, Borel summability, Borel transform and Borel sum will be needed, we recall them briefly following \cite{TH} (see also the related papers \cite{TH1}, \cite{TH2}.)\\
 For any  $\kappa,\, T> 0$ we define ${\cal P}_{\kappa,\,T}$ as the set of complex-valued functions $f$ of a complex variable $t$ which are analytic on ${\tilde D}_{T}:=\{ z\,\in\,\C\,\mid Re{1 \over z} > {1 \over T}\}$ (the open disk of center ${T \over z}$ and radius ${T \over z}$ ) and such that there exist sequences $(a_k,\,R_{k}),\,k\,\in\,\N_0,$ of complex analytic functions on ${\tilde D}_{T}$ such that:
 \begin{enumerate}
 \item For every $r \in\,\N, k, T>0$ and $t\,\in\,{\tilde D}_{T} $ it holds that
 \begin{equation}
  f(t)=a_0+a_1\,t+\cdot\cdot\cdot+a_{r-1}\,t^{r-1}+R_r(t).
 \end{equation}
 We then say that $f \,\in\,{\cal P}_{\kappa,\,T} $ with coefficients $\{a_k\}$ and remainders $R_k,\, k\,\in\,\N_0$.\\
 
 \item For every $\bar{\kappa} < \kappa,\,\, \bar{T}< T $ there exists $K>0$ (depending on $f, r$) such that for every $r\,\in\,\N_0$ and $t\,\in\,{\tilde D}_{T}$ the following estimates hold
 \begin{equation}
 \mid R_r(t) \mid \leq K\, {{r!} \over {\bar{\kappa}}^r } \,\mid t\mid^r.
 \end{equation}
 
 \item For any  $\kappa,\, T> 0$ we define ${\cal Q}_{\kappa,\,T}$ to be the set of functions $g$ such that $g$ is complex analytic on ${\tilde S}_{\kappa}:=\{ z\,\in\,\C\,\mid d(z,\,[0,+\infty))< \kappa   \}$ and for every $\bar{\kappa} < \kappa,\,\, \bar{T}< T $ there exists $\tilde{\kappa} $ such that for every $\tau\,\in\,{\tilde S}_{\bar{\kappa}}$ the following estimate holds
  \begin{equation}\label{KAS}
  \mid g(\tau) \mid \leq K_g\, \exp({\mid \tau \mid}/{ \bar{T}}),
  \end{equation}
 for some constant $K_g$ (depending on $g$), with $d(z,A):=inf\{\mid z-w \mid,\,w\,\in\,A \},\,\,A$ a subset of $\C$.\\
The following theorem goes back to Nevanlinna, who in turn improved a theorem of Watson, see \cite{SO}, \cite{Rez} .

\begin{theo}\label{ABDEL}
If\, $f \in {\cal P}_{\kappa,\,T}$, for some $\kappa,\, T> 0$, with coefficients $\{a_k\}$ and remainders $\{ R_k\}\,,\,k\,\in\,\N_0$, then 
 \begin{equation}\label{NA}
 B(f)(\tau):=\ds_{r=0}^\infty {{a_r} \over {r!}}\,\tau^r
 \end{equation}
converges absolutely for  $ \tau \in \R$ and admits an analytic continuation $g$ to ${\tilde S}_{\kappa}$ which belongs to ${\cal Q}_{\kappa,\,T}$. Moreover, if $h$ is a function in ${\cal Q}_{\kappa,\,T}$ then 
\begin{equation}\label{NAW}
 f_h(t):={1 \over t}\dint_0^\infty h(\tau)\,e^{-{\tau \over t}}\,{d\tau}
 \end{equation}
is in ${\cal P}_{\kappa,\,T},\,$ for any $t \in\, {\tilde{D}}_{T}$. In particular, if $h=g$ then $f_g$ coincides with $f$.
\end{theo}
\begin{rem}
\begin{enumerate}
\item  $f_g$ is the Laplace transform of $g$ and $B(f)$ as defined in (\ref{NA}) is the Borel transform of $f$.\\
\item If $f \in\,{\cal P}_{\kappa,\,T} $, then the coefficients $a_r$ are uniquely determined and their knowledge suffices to recover $f$ through the analytic continuation $g$ of its Borel transform  $B(f)$ as given by (\ref{NA}). (Thus $f$ is then the Laplace transform of its Borel transform $B(f)$, i.e., $g$ satisfies $Q_{k, T}$ and for $t \in {\tilde{D}}_{T}:$
\begin{equation}
f(t)=f_g(t),
\end{equation}
for $g=B(f)$) (and $B(f)$ only depends on the $a_r\,r\,\in\,\N_0$, by (\ref{NA})).

\end{enumerate}
\end{rem} 
Finally let us define Borel summability and Borel sum.\\
Let $\tilde{a}_r,\,r\,\in\,\N_0$ be complex numbers. A formal power series $\ds_r \tilde{a}_r\,t^r,\,t\,\in\,\C$ is said to be Borel summable if there exist $\kappa,\, T> 0$ and a function $f \in {\cal P}_{\kappa,\,T}$ with  coefficients $a_r,\,r\,\in\,\N_0$ s.t. for every
 $r\,\in\,\N_0,\,a_r=\tilde{a}_r.\, $ 
 $f$ is then called Borel sum of the formal power series $\ds_r \tilde{a}_r\,t^r$ in ${\tilde{D}}_{T}$.

\end{enumerate}  
                   
We have from \cite{TH} (Theorem 3.1)(restricting ourselves to the case of a complex-valued $\mu_V$ (instead of the matrix-valued as assumed in \cite{TH}, since it suffices for the applications that we have in mind, see. Section 4,) the following result:  
\begin{theo}(Harg\'e, \cite{TH1}) \label{PA}
 Let $V$ be  as in Hypothesis 2.1, and assume in addition that $\dint e^{a\,\mid\xi\mid^2}\mid\mu_V
\mid(d\xi)< \infty , \, $  for some $a > 0$ (this is a regularity assumption on $V$).\\
Let $p_t(x,y)\,,t \geq 0,\,x,y\,\in\,\R^d$ be the kernel of the strongly continuous contraction semigroup $e^{t(\Delta-V)}$ in $L^2(\R^{d})$.
\begin{enumerate}
\item 
Let , for $t > 0, \,\tilde{p}_t(x,y):=(4\,\pi\,t)^{d \over 2}\,e^{{(x-y)^2 \over {4\,t}}}\,p_t(x,y) $. Then for every $x,y\,\in\,\R^d, $ $\tilde{p}$, as a function of $t$, has a Borel transform $B(\tilde{p})(\tau, x, y)$ defined for $\tau \,\in\,\R$ by (\ref{NA}) (with $a_r$ the coefficients of the expansion of $\tilde{p}_t$ in powers of $t$), which is analytic in $(\tau,\,x,\,y)$ on $\mathbb{C}^{1+2d}$.\\
\item
Let $S_{\kappa}:=\{ z\,\in\, \mathbb{C}\,| \,\,\,\, | Im\,z^{1 \over 2}|^{2}\,\,< \kappa \}$, (where for any $z\,\in\,\C$ of the form $z=\mid z \mid\,e^{i\,\theta},\, \theta \,\in\, (-\pi,\,\pi]$ we denote $z^{1 \over 2}=\mid z \mid^{1 \over 2}\,e^{i\,{\theta \over 2}} $)  , for a given $\kappa > 0$ (this is a parabola which contains $\tilde{S}_{\kappa}$, as defined just before (\ref{KAS}) ). Then  for every $(\tau, x, y)\, \in\, S_{\kappa}\,\times\, {\mathbb{C}_R^{2d}}$, with ${\mathbb{C}_R^{2d}}:=\{ (x ,y)\,\in\,\mathbb{C}^{2d},\,  | Im\,x| < R,\,\, | Im\,y|< R\},$ for some given constant $ R> 0$, we have $|B(\tilde{p})(\tau,\,x,\,y)|\leq \exp(C\,|\tau|^{1\over 2}),$ with 
\begin{equation}C := 2\,e^{{\kappa \over a}}\Big(\dint_{\R^d}\exp({a \over 2}\xi^2+ R\mid \xi \mid) \mid \mu \mid(d \xi) \Big)^{1 \over 2},
\end{equation}
 (using the same symbol $B(\tilde{p})$ for the analytic continuation of $B(\tilde{p})$ from $\R \times \C^{2d}_R$ to $S_{\kappa} \times \C^{2d}_R$ ).\\
Moreover,\, for any $(x ,y)\,\in\,\mathbb{C}_R^{2d}$,\, the small time expansion of\,\, $\tilde{p}_t(x,y)=(4\,\pi\,t)^{d \over 2}\,e^{{(x-y)^2 \over {4\,t}}}\,p_t(x,y), t > 0,$ in powers of $t$ exists, $\tilde{p}_t $ is in ${\cal P}_{\kappa,\,T},$ (with ${\cal P}_{\kappa,\,T}$ as defined before Theor. \ref{PA}), is Borel summable and its Borel sum is equal to $\tilde{p}_t$.
\end{enumerate}
\end{theo}
\begin{proof}
 The proof is provided in \cite{TH1}.
 \end{proof}
\begin{theo}(Harg\'e) \label{PAP}
Let $\omega \,\in\,\C,\,V$ as in Theorem \ref{PA}. Let $u(t, x, y),\,t >0,\, x,y\,\in\,\R^d$ be the solution of 

\begin{equation}
 \partial_t\,u(t, x, y)=(\Delta_x-{{\omega^2 \over 4}}\,x^2)\,u(t, x, y)-V(x)\,u(t, x, y),\,\,x\,\in\,\R^d,
\end{equation}

with $\dlim_{t \longrightarrow\,0} u(t, x, y) = \delta(x-y) $ in the distributional sense. $u(t, x, y), t > 0 $ is then, for $\omega \,\in\,\R,$ the density of the kernel of the operator $e^{t(\Delta-W)}$ in $L^2(\R^d)$ with $W(x)=V(x)+{{\omega^2 \over 4}}\,x^2$. Then  for any $r \in \N_0$, there exists $T> 0$ and analytic functions $a_r$ on $\mathbb{C}^{2d}$ and $R_r$ on $D_T^{+} \times \mathbb{C}^{2d},$   with\\ $D_T^{+}:= D_T \cap \C^{+} ,\,\,\,D_T:= \{ z\,\in\, \mathbb{C}\,| \,\,\,\, | z|\,< T \}$, such that 
\begin{equation}\label{IS}
u(t, x, y)= (4\,\pi\,t)^{-{d \over 2}}\,e^{-{(x-y)^2 \over {4\,t}}}\,\Big( 1+a_1(x,y)\,t+\cdot\cdot\cdot+a_{r-1}(x,y)\,t^{r-1}+R_r(t,x,y)\Big),
\end{equation}
for every $r\, \in \N, \,t\,\in\,D_T^{+},\, (x,y)\,\in\, \mathbb{C}_R^{2d}.$\\
Moreover, for each $R> 0$, there exist $ K,\,\kappa >0$ such that 
\begin{equation}\label{IS1}
 \mid R_r(t, x, y) \mid \leq K\, {{r!} \over {{\kappa}}^r } \,\mid t\mid^r,
 \end{equation}
for every $r\,\in\,\N_0,\,t\,\in\, D_T^{+},\,(x,y)\,\in\, \mathbb{C}_R^{2d}$.\\
The expansion for\, $u(t,x,y)\,(4\,\pi\,t)^{{d \over 2}}\,e^{{(x-y)^2 \over {4\,t}}}:=v(t, x, y)$ (obtained from the one in (\ref{IS}) multiplying both members by $(4\,\pi\,t)^{{d \over 2}}\,e^{{(x-y)^2 \over {4\,t}}} $), as a function of $t$, constitutes an asymptotic series in powers of $t,\, \sum_{j=0}^{r-1}\, a_j(x, y)\,t^j\,+ R_r(t, x, y), $ with $a_0\equiv 1,$ that has a Borel transform belonging to ${\cal Q}_{\kappa, T}$, for any $x, y \in \mathbb{C}_R^{2d}.$ The Laplace transform (\ref{NAW}) of this Borel transform coincides with $v(t,x,y)$. The asymptotic series in powers of $t$ constituting the above small time expansion of $v(t,x, y)$ is Borel summable to $v(t, x, y)$.
\end{theo}
\begin{proof}
 For the proof we refer to  \cite{TH}, Theor. 3.5.
 \end{proof}
\begin{rem}
\begin{enumerate}
\item If $\omega\,\in\,\R, u(t,x,y)$ is the kernel of an analytic semigroup, the one given by the density of the kernel of $e^{t(\Delta-W)}$ in $L^2(\R^d)$.
\item For $\omega \in i\R,$ \cite{TH} has also provided a uniqueness result for analytic semigroups related to $u$.
\end{enumerate}
\end{rem}
 \section{Borel summation of the transition kernel for the SDE with additive noise}
 Following Theorems \ref{PA} and \ref{PAP} and the assumptions on $V$ given in the previous section we have the following small time expansion of (a quantity simply related to) the kernels giving the transition semigroup for the solution process of the SDE (1) with $\sigma \equiv {\bf 1}$.
 
 \begin{theo} \label{PAAP}
 Consider the stochastic differential equation
 \[\displaystyle  \left\{\begin{array}{lll} dX_{t}=\beta\, (X_t)\,dt+dW_{t}\\ \displaystyle x_0=x^0,\,\,x^0 \in \mathbb{R}^d\\ \end{array}\right.\]
 
 where the drift $\beta$ is a gradient field such that
 $\beta(x)=2\psi^{-1}\,(x)\nabla \psi\,(x),$ for a primitive smooth
 function $\psi$ of the form \begin{equation} \label{TUN}
 \displaystyle \psi{(x)}=c_{\varphi}\,.\,\varphi(x)\,
 \frac{e^{-\omega\frac{| x |^2}{4}}}{(2\pi\,{2 \over {\omega}})^{d/2}},
 \end{equation}
 
 where $\varphi > 0$ is a smooth function for $\mathbb{R}^d$ to
 $\mathbb{R}_{+}, \,\,\omega^2>0,$ such that $\displaystyle
 {\mid\varphi\mid^2}\,e^{-{\omega\over 2}\mid\cdot\mid^2)}\in L^1\,(\mathbb{R}^d)$.\\ $c_{\varphi}$ is the normalization constant given by
 \begin{equation} 
 \displaystyle c_{\varphi}= \left(\int_{\mathbb{R}^d} \varphi^2(x) \frac{e^{-{\omega\over 2}|x|^2}}{\left(2\pi\frac{2}{\omega}\right)^{d}}\,\,dx\right)^{-1},
 \end{equation}
 so that $\displaystyle \int_{\mathbb{R}^d}|\psi^2|(x)\,dx=1.$\\ 
   Set $\beta_{\varphi}(x)=2\,\varphi^{-1}(x)\nabla \varphi (x)$ and
 assume $V_{\varphi}(x):={1 \over 4}\mid\beta_{\varphi}(x)\mid^2+ {1 \over 2}\, div\,\beta_{\varphi}(x)$
 satisfies the assumption $2.1$, and
 \begin{equation} 
 V_{\varphi}(x)=-\int_{\mathbb{R}^d}\,e^{i xy}\, \, \, \mu_{V}\, \, (dy), \end{equation}
 for some bounded symmetric complex measures $\mu_{V}$ on
 $\mathbb{R}^d,$ such that (as theorem \ref{PA}, for $V=V_{\varphi})$,\,\,\,
 $\displaystyle \int_{\mathbb{R}^d}e^{a|y|^2}\,
 |\mu_{V}|\,(dy)<\infty,$ for some $a>0$.\\
  Then  the following properties hold:\\
 
 (i) The unique solution process $X_t$ of (1) is a diffusion process
 with drift $\beta(x)=\beta_{\varphi}(x)-\omega\,x,$ and invariant
 symmetrizing probability measure $\nu(dx)=\psi^2\,(x)\,dx.$\\
  The corresponding Markov transition semigroup $p_t, t>0$ has a kernel with
 density $k(t,x,y),\,t>0,\, x,y\in\mathbb{R}^d$, such that
 $(p_t\,f)\,(x)=\int_{\mathbb{R}^d}\, k(t,x,y)\, f(y)\,dy,$ for any
 bounded $C_0\,(\mathbb{R}^d)$- function $f$. $p_t$ is the symmetric
 Markov $C_0$ semi-group in $L^2(\mathbb{R}^d,\, \nu)$ associated
 with the classical Dirichlet form $\varepsilon_{\nu}$ given by $\nu$,
 i.e.
 \begin{equation}
 \varepsilon_{\nu}\,(f,g)=\int_{\mathbb{R}^d}\, \nabla f
 \cdot \nabla g\, d\nu.
 \end{equation} 
 The generator $L_{\nu}$ of $p_t$ is the self-adjoint operator in
 $L^2(\mathbb{R}^d, \nu)$ given on smooth function $f$ of compact
 support by $L_{\nu}\,f(x)=(\triangle+\beta_{\varphi}(x)\cdot \, \nabla\,-
 \omega\, x \nabla)\,f(x).$\\ The operator $-L_{\nu}$ is unitary equivalent with the
 self-adjoint operator
 $H=-\triangle+V_{\varphi}(x)-{\omega \over 2}\,x\,\beta_{\varphi}(x)+\frac{\omega^2}{4}|x|^2-{d \over 2}\,\omega$ acting in $L^2\, (\mathbb{R}^d).$ 
 
 (ii) One has $k(t,x,y)=\psi^{-1}(x) \, u (t,x,y)\, \psi(y),$ with
 $u(t,x,y)$ as in Theorem \ref{PAP} (with $\omega >0$). $u(t, x, y)$ satisfies the
 asymptotic expansion given by (\ref{IS}) and (\ref{IS1}), and consequently $k(t, x, y)$ has a corresponding asymptotic expansion, given by (\ref{IS}), (\ref{IS1}) multiplied by $\psi^{-1}(x) \,\psi(y) $.\\
  The quantity
 \begin{equation}\label{HI}
  \tilde{k}(t,x,y):=(4\,\pi\,t)^{d/2}\, e^{{(x-y)^2} \over {4\,t}}\, k(t,x,y), 
 \end{equation}
 
 which we call modified kernel function for the solution process
 $X_t$ to (1), has, as a function of $t$, an asymptotic expansion in
 powers of t, its Borel transform belonging to $Q_{\kappa,T},$ for any $x,y
 \in\mathbb{C}^{2d}_{R}, \, R>0$( with $Q_{\kappa,T}$ as in point 3 of Hypothesis 2.1 ). The asymptotic expansion for $\tilde{k}(t,x,y)$ is given by\,\,\, $\psi(x)\Big[ 1+ \ds_{j=1}^{r-1} a_j\,\psi(x)(1+ \ds_{j=1}^{r-1} a_j(x ,y)\,t^j+\,R_r(t,x,y))\Big]\,\psi^{-1}(y),$ with the $a_j,\,j=1,...,r-1\,$  as in (\ref{IS}).\\
 The  Laplace transform of this Borel transform (i.e. of $k(t, x, y)$) coincides with the
 function given in (\ref{HI}), which constitutes then the Borel sum of the asymptotic
 expansion in powers of $t$ for the modified transition kernel $\tilde{k}(t,x,y)$ for
 the solution of the SDE.
  \end{theo}
 \begin{proof}
 i) The proof uses essentially the theory of Dirichlet forms and the
 regularity assumption on $\psi$, see, e.g. \cite{Al}, \cite{AHKS}, \cite{FOT}.\\
  Let us look at the Dirichlet form $\varepsilon_{\nu}$ defined on
 $C^{\infty}_{0}\, (\mathbb{R}^d)$. By integration by parts we have
 for $f,g \in C^{\infty}_{0}\, (\mathbb{R}^d):$
 \begin{equation}
 \varepsilon_{\nu}\, (f,g)=\int \nabla f \cdot \nabla g\, d\nu=-\int f
 \triangle \,g \, d\nu-\int f \, \beta_{\psi}\nabla g \, d\nu
 ,\end{equation}
  where \begin{eqnarray}\label{ZIT} \beta_{\psi}(x)&=& 2 \, \nabla \ln\, \psi(x)\nonumber\\&=& 2\nabla \ln\,
 \varphi(x)+2\nabla \ln c_{\varphi} \frac{e^{-\omega{|x|^2 \over 4}}}{(2\pi
 \frac{2}{\omega})^{d\over 2}}\nonumber\\&=& \beta_{\varphi}(x)-\omega\,x,
 \end{eqnarray}
 where we used the definition of $\psi, \beta_{\psi}$ and $\beta_{\varphi}$. Hence we found the given expression for $L_{\nu}.$ The function identically equal to ${\bf 1}$ is the domain of the positive operator $L_{\nu}$ and we have $L_{\nu} {\bf 1}(x)=0$ for all $x \in \R^d.$\\
  The unitary equivalence comes from the ground state transformation
 $L^2(\mathbb{R}^d;\nu)\rightarrow L^2(\mathbb{R}^d)$ given by $f \in
 L^2 (\mathbb{R}^d;\nu)\rightarrow f\psi \in L^2 (\mathbb{R}^d),$
 which on dense domain maps $-L_{\nu}$ into $H=\psi{(-L_{\nu})} \,\psi^{-1}$ and yields, 
 by computation for smooth $f$:
 \begin{equation}\label{HAD}
-\psi^{-1}H\,\psi \,f=L_{\nu}\,f=(\triangle +\beta_{\psi}\cdot
 \nabla)f\,,
 \end{equation}
 where we used the expression for $L_{\nu}$ on smooth functions.\\
From (\ref{HAD}) we deduce, multiplying both sides by the smooth function $\psi$, for $f$ smooth:
\begin{equation}\label{HID}
-H\,(\psi \,f)=\psi(\triangle +\beta_{\psi}\cdot
 \nabla)f\,.
 \end{equation}
We claim that, on smooth functions,\begin{equation} H=-\triangle+\tilde{V},\end{equation} with 
\begin{equation}\label{ZI}
\tilde{V}(x):= {1 \over 4}\mid\beta\mid^2_{\psi}(x)+{1 \over 2}\,\mbox{div}\,\beta_{\psi}(x),
 \end{equation}
  In fact then using Leibniz formula and the definition of $ \beta_{\psi}$:
 \begin{eqnarray} (-\triangle+\tilde{V} )\,(\psi\,f)&=&-\triangle
 \psi\cdot f-2 {\nabla
 \psi} \nabla f-\psi\,\triangle f +\frac{\mid\nabla \psi\mid^2}{\psi^2}\,\psi f+\,\mbox{div}\,\frac{\nabla \psi}{\psi}\,\psi f
 \nonumber\\&=& -\triangle \psi\,f-2\, {\nabla \psi}\,\nabla f-\psi\triangle f+  \frac{\mid\nabla
  \psi\mid^2}{\psi}\,f+ \frac{\psi\triangle f-\mid\nabla
    \psi\mid^2}{\psi^2}\,\psi f,\,\,\,\,\,\,\,\,\,\,\,\,\,\,\,\,\,\,\,\,   
 \end{eqnarray}
 which is seen to be equal to the right hand side of (\ref{HID}).  This shows that (\ref{HAD}) holds.\\
 From (\ref{HID}) we see, taking $f= {\bf 1},$ that $ H\psi=0,$ or using  $H=-\triangle +\tilde{V}$, that $\tilde{V} $ is also expressed by $ \tilde{V}= \frac{\triangle \psi}{\psi}.$\\
Computing in our special case of $\psi$ of the form (\ref{TUN}) and using (\ref{ZIT}) and (\ref{ZI})
 we get 
 \begin{equation}
  \tilde{V}(x)={1 \over 4} \mid\beta_{\varphi}(x)- \omega x \mid^2+ {1 \over 2}\,\mbox{div}\,\beta_{\varphi}(x)-\frac{\omega\,d}{2}\,,\,\,(since\,\,\, \mbox{div}\, x =d).
 \end{equation}
 Hence 
 \begin{equation}
   \tilde{V}(x)=V_{\varphi}(x) + {1 \over 4}\,\omega^2 \mid x\mid^2-{\omega \over 2}\beta_{\varphi}(x)\,x -{\omega \over 2}\,d,
  \end{equation}
 where 
  \begin{equation}
  {V}_{\varphi}(x):={1 \over 4} \mid\beta_{\varphi}(x)\mid^2+ {1 \over 2}\,\mbox{div}\,\beta_{\varphi}(x)
  \end{equation}
 is as stated in the theorem.\\
  Then this proves i), except for the statements on the smoothness of
 the density of the kernel of $p_t,$ for $t >0,$ that follow from the
 general theory of strictly elliptic generators of diffusions, with smooth
 coefficients, and the theory of Dirichlet forms.\\
 ii) The results on the asymptotic expansion for $k$ follow from the one on $v$ as given in theorem \ref{PAP}.
 \end{proof}
 \begin{rem}
i) From the proof it follows that we have also $\tilde{V}=\frac{\triangle \psi}{\psi} $ and $ {V}_{\varphi}=\frac{\triangle \varphi}{\varphi}.  $\\
ii) The unitary transformation $L^2(\mathbb{R}^d)\rightarrow
L^2(\mathbb{R}^d, |\psi|^2\,dx)$ given in the proof of Theorem 3.1,
defined by
\begin{equation} 
f\in L^2 (\mathbb{R}^d)\rightarrow \frac{f}{\psi} \in L^2 \, (\mathbb{R}^d,
|\psi|^2\,dx),
\end{equation} is called ground state transformation, see, e.g,
\cite{Al, AHKS, AlMa, RS}. It permits to transform the operator\,
$-(\triangle + \beta \cdot \nabla)$ with $\beta$ a gradient field( as
coefficient of the first order differential term) into the operator  $-\triangle
+\tilde{V},$ (with a potential term $\tilde{V},$ a multiplication operator, i.e. a zero order differential operator). In probability
theory it is often also called Doob's $h$-transform.
\end{rem}
\section{Some remarks on applications}
We can apply the above theorem \ref{PAAP} to stochastic differential equations descriving  a drift term perturbations of  Brownian motion where the drift is a gradient
field of the form\, $\beta(x)=\beta_{(V)}(x)-\omega\,x,\, \omega >0.$ $\beta_{(V)} $ satisfies the smoothness assumption for $V= ({1 \over 4}\mid\beta\mid^{2}+ {1 \over 2}\mbox{div}\, \beta)$
given in hypothesis 2.1 in terms of the measure $\mu_{V}$ of which $V$ is the
Fourier transform. Such models can be looked upon as perturbations of the
Ornstein-Uhlenbeck (velocity) process by a drift term given by $\beta_{(V)}$.
They have numerous applications e.g. in quantum mechanics and
(quantum) statistical mechanics, see, e.g \cite{AHKS}.\\
 For $d=1$ there are applications in mathematical finance(
smooth perturbations by a drift term of the Vasicek model, for the latter model, see,
e.g. \cite{Vas}).\\
Our results gives essentially small time Borel summable expansions in power series for the
transition probabilities in such models, permitting to compute all
terms, with control on remainders.\\
In areas like mathematical finance, (see, e.g. \cite{FoS, PJA, GHLOW, SA, Wafa}),  models  with multiplicative noise also play an extensive role. They are usually stated for real valued processes satisfying SDEs of the form (1) with $d=1$ but with $\sigma$ not constant  (e.g. models of the type of
Black-Scholes, Pearson's or CIR models).\\
Some of these models can also be covered, by an adaptation of methods in this paper, in the sense of getting for them Borel summable asymptotic expansions for their solutions processes. Let us describe the main
idea, leaving for future publications more detailed applications. The
idea consists in a use of time change in the form of a Lamperti
transformation from an equation of the form (1) for $ X_t$, with $\sigma$
constant positive, to an equation for a transformed process $\tilde{X}_t$
satisfying an equation of the form
$d\tilde{X}_t=\tilde{\beta}_{\sigma}\, (\tilde{X}_{t})\,dt+
dB_{t}$ where $B_t$ is a standard Brownian motion and
$\tilde{\beta}_{\sigma}$ is a new drift obtained by using a Lamperti
transforms $\gamma$ defined by $\displaystyle \gamma_{s_0}(s):=
\int^s_{s_0}\, \frac{dy}{\sigma (y)}, \, {s, s_0}\, \in \mathbb{R},\, s_0< s$ (we assumed $\sigma$ to be smooth strictly positive and such that
$\displaystyle \frac{1}{\sigma} \in L^1_{loc}\, (\mathbb{R}), \,
\frac{1}{\sigma} \notin L^1 (\pm\infty),\, \beta'-\beta
\frac{\sigma'}{\sigma}-\frac{1}{2}\,\sigma\sigma''$ bounded,
$|\beta|,\sigma$ linearly bounded.\\
 By results in Lamperti \cite{La}(see also, e.g.,[\cite{KaShre},Ex. 2.20], \cite{MMH}), we then have
$\tilde{\beta}_{\sigma}(x):=\displaystyle
\frac{\beta(s)}{\sigma(s)}-\frac{\sigma'(s)}{2},$ with $s$ and $x$
related by $\gamma_{s_0}(s)=x$ and the transformation $ x \rightarrow  s$ being invertible, under the stated
assumptions.) The transition function $p^X$ for the process $X$ is
then related to the one for $p^{\tilde{X}}$ by $p^X\,
(t,x,\tilde{x})=\sigma(\tilde{s})^{-1}\, p^{\tilde{X}}\,
(t,s,\tilde{s}),$ with
$\tilde{\gamma}_{s_0}(\tilde{s})=\tilde{x},\,\,s,\tilde{s}, x , \tilde{x}$ being
real variables. Provided $\tilde{\beta}_{\sigma}$ satisfies the
assumptions stated for $\beta$ in theorem \ref{PAAP}, we can apply this theorem to $\tilde{X}$ and then obtain a Borel summable small time asymptotic expansion for $p^{\tilde{X}}$,
thus also for the transition probability kernel for the process
$X$. In this way we obtain the following:
\begin{theo} \label{TOUN} Suppose $\sigma$ satisfies the assumption $\sigma>0,\,\sigma$
smooth, ${1 \over \sigma}\in L^{1}_{loc}(\mathbb{R}), \, {1 \over \sigma}\not\in
L^1(\pm \infty),$ moreover $\beta'-\beta
\frac{\sigma'}{\sigma}-\frac{1}{2}\sigma\sigma''$ bounded, and
$|\beta|,\, \left|\frac{\beta}{\sigma}-\frac{\sigma'}{2}\right|$ and $
\sigma$ linearly bounded at infinity, and finally ${1 \over 4}\mid\tilde{\beta}_{\sigma}(x)\mid^2+ {1 \over 2}\,div\, \tilde{\beta}_{\sigma}(x)$ required to be of the form $\displaystyle
V_{\varphi}(x)+\frac{1}{4}\,{\omega}\,x^2-{\omega \over 2}\,{\beta}_{\varphi}(x)x-{d \over 2}\omega  $ for some ${\omega}>0,\, \varphi>0,$ where $\varphi, V_{\varphi},  {\beta}_{\varphi}$ are as in Theorem \ref{PAAP}. Then $p^X$ has a Borel summable small time asymptotic expansion in powers of $t$.
\end{theo}
\begin{rem}
All assumptions in Theorem \ref{TOUN} are e.g if the following conditions are satisfied:\\
i) for $\sigma>0 $ smooth of the form
$\sigma(x)=\sigma_0\, +c_1 x+ o(x)$ for some $\sigma_0 >0$ sufficiently
small and $c_1>0$ as $x \rightarrow 0$,  and $\sigma(x)= c_2\,x^{\alpha}+o(x^{\alpha} ),$ for some $ \alpha <1, c_2 > 0$ as $\mid x\mid
\rightarrow \infty$.\\
ii) In addition to $\beta$ being $\,C^1$ and $ \mid \beta\mid$ linearly bounded at infinity, one requires that  $\beta(x) \sim x^{\beta},\, \beta >1$ as $\mid x\mid \rightarrow 0 $.\\
We postpone further discussions of such conditions and applications to forthcoming work, relating also to previous work by \cite{Lu} on Borel summability in the
framework of Black-Scholes type models.
\end{rem}
\newpage
\noindent{\bf Acknowledgements:} This work was supported by King
Fahd University of Petroleum and Minerals under the project SF181-MATH-476. The authors gratefully acknowledge this support.

\medskip
\vskip8cm
\begin{flushleft}
\footnotesize
{\it Boubaker Smii} \\
Dept. Mathematics, King Fahd University of Petroleum and Minerals, \\
Dhahran 31261, Saudi Arabia\\
\medskip
{\it Sergio Albeverio}  \\
 Dept. Appl. Mathematics, University of Bonn,\\
HCM;  BiBoS, IZKS, KFUPM(Dhahran); CERFIM (Locarno)\\
\medskip

\medskip
 E-mail: boubaker@kfupm.edu.sa  \\
\hspace{1,1cm}albeverio@uni-bonn.de \\

 \end{flushleft}
\end{document}